\documentclass[a4paper,12pt]{article}
\usepackage{amsmath, amsthm, amstext}
\usepackage{amssymb, array, amsfonts}
\usepackage[utf8]{inputenc}
\usepackage{graphicx}

\numberwithin{equation}{section}

\def \al{\alpha}
\def \be{\beta}
\def \ga{\gamma}
\def \de{\delta}
\def \er{\varepsilon}
\def \ze{\zeta}
\def \ka{\varkappa}
\def \la{\lambda}

\def \te{\theta}
\def \ph{\varphi}
\def \oo{\omega}

\def \G{\Gamma}
\def \D{\Delta}
\def \L{\Lambda}
\def \O{\Omega}

\def \C{\mathbb{C}}
\def \N{\mathbb{N}}
\def \R{\mathbb{R}}
\def \Z{\mathbb{Z}}

\def\n{\nabla}
\def\dd{\partial}

\def\1{1\!\!\!\!1}

\def\mat{\operatorname{Mat}}
\def\mes{\operatorname{mes}}

\def\spec{\operatorname{spec}}

\def\im{\operatorname{Im}}
\def\re{\operatorname{Re}}
\newcommand{\<}{\langle}
\renewcommand{\>}{\rangle}

\theoremstyle{plain}
\newtheorem{theorem}{\bf Theorem}[section]
\newtheorem{lemma}[theorem]{\bf Lemma}

\newtheorem{cor}[theorem]{\bf Corollary}

\theoremstyle{definition}

\theoremstyle{remark}
\newtheorem{rem}[theorem]{\bf Remark}

\newtheorem{conj}[theorem]{\bf Conjecture}

\renewcommand{\le}{\leqslant}
\renewcommand{\ge}{\geqslant}

\renewcommand{\qed}{\vrule height7pt width5pt depth0pt}

\title{Inequalities between Dirichlet and Neumann eigenvalues in large dimensions}
\author{N.~D.~Filonov}
\date{}
\begin{document}
\maketitle

\begin{abstract}
Let $\O$ be a bounded domain in $\R^d$.
Denote by $\la_k$ (resp. $\mu_k$)
the eigenvalues of the Laplace operator in $\O$ with Dirichlet (resp. Neumann)
boundary conditions. 
Denote by $\Psi = \Psi (d,k,\O)$ the shift of indices in the inequality 
$\mu_{k+\Psi} \le \la_k$.
We are interested to describe the behaviour of $\Psi$ for large $d$.
We prove that
a) $\Psi (d,1,\O) \ge C (e/2)^d$ for all domains $\O$; and
b) $\Psi (d,k,\O) \ge C (e/2)^d$ for all $k$ and all convex domains $\O$.
\footnote{Keywords: 
Laplace operator, Dirichlet problem, Neumann problem,
inequalities for eigenavlues, convex domains}
\end{abstract}

\tableofcontents

\section{Introduction}
Let $d \ge 2$, $\O \subset \R^d$ be a bounded domain 
such that the embedding $W_2^1 (\O) \subset L_2 (\O)$ is compact,
where $W_2^1 (\O)$ is the Sobolev space.
We consider Dirichlet and Neumann problems for the Laplace operator in $\O$:
$$
\begin{cases}
- \D \ph_k = \la_k \ph_k, \\
\left.\ph_k\right|_{\dd\O} = 0,
\end{cases}
\qquad
\begin{cases}
- \D \psi_k = \mu_k \psi_k, \\
\left.\frac{\dd\psi_k}{\dd\nu}\right|_{\dd\O} = 0.
\end{cases}
$$
Spectra of these two problems are discrete,
the eigenvalues form two sequences
\begin{equation*}
0 < \la_1 < \la_2 \le \la_3 \le \dots, \qquad \la_k \to + \infty, 
\end{equation*}
\begin{equation*}
0 = \mu_1 < \mu_2 \le \mu_3 \le \dots, \qquad \mu_k \to + \infty.
\end{equation*}
We take multiplicity into account.
In 1952, Kornhauser and Stakgold proved in \cite{KornStak} on the physical level of rigor
the inequality
\begin{equation}
\label{mu2}
\mu_2 < \la_1
\end{equation}
in two-dimensional case.
They needed it to investigate the spectrum 
of electromagnetic waves in a three-dimensional cylinder.
\footnote{For the connection between this inequality and 
the spectrum of the Maxwell operator, see, for example, 
\cite{F18, F24}.}
In his paper \cite{P52}, published in the same journal issue as \cite{KornStak},
Pólya used their ideas to provide a rigorous proof of the inequality \eqref{mu2}.
His proof works in any dimension.
This work by Pólya caused a stream of papers 
devoted to the improvements of this inequality,
see \cite{Payne55, Av86, LW86, Fr91, F04, R25, TZh, mu3}.
In particular, Levine and Weinberger proved in \cite{LW86} the inequality
\begin{equation}
\label{11}
\mu_{k+d} \le \la_k \qquad \forall \ k \in \N, \ 
\text{for all convex domains} \ \O \subset \R^d.
\end{equation}
Friedlander proved (see \cite{Fr91, F04}) the inequality
\begin{equation}
\label{muk+1}
\mu_{k+1} < \la_k
\end{equation}
in all dimensions $d \ge 2$.
Rohleder proved in \cite{R25} the inequality
$$
\mu_{k+2} \le \la_k 
$$
for simply-connected Lipschitz domains in two-dimensional case.
The inequality
$$
\mu_3 < \la_1
$$
was proved in \cite{TZh} for Lipschitz domains if $d=2$,
and in \cite{mu3} for all domains if $d \ge 3$.

The last sentence of \cite{LW86} says

{\it perhaps \eqref{11} can be replaced by a better inequality of the form
$$
\mu_{\Phi(d,k)} < \la_k
$$
for convex $d$-dimensional domains.}

We show that in large dimensions this inequality indeed can be improved,
see Theorem \ref{t14} below.

Introduce notations
$$
\Phi (d, k, \O) := \# \left\{ j : \mu_j (\O) \le \la_k (\O)\right\},
$$
$$
\Psi (d, k, \O) := \Phi (d, k, \O) - k
$$
in such a way that the inequality
$$
\mu_{k + \Psi(d,k,\O)} \le \la_k
$$
holds true by definition.
In these terms Friedlander's inequality \eqref{muk+1} means 
\begin{equation}
\label{*}
\Psi (d,k,\O) \ge 1
\end{equation}
for all $k$ and all domains,
Rohleder's inequality means
$$
\Psi (2,k,\O) \ge 2
$$
for simply connected planar domains, 
and Levine-Weinberger's inequality \eqref{11} means 
$$
\Psi (d,k,\O) \ge d
$$ 
for all $k$ and all {\it convex} domains.

Our first result concerns the cases $k=1$ and $k=2$.

\begin{theorem}
\label{t11}
The inequalities
\begin{equation}
\label{12}
\Phi (d,1,\O) \ge \frac{(j_{d/2-1})^d}{(d+2) 2^{d-1} \G(d/2+1)^2}
\end{equation}
and
\begin{equation*}
\Phi (d,2,\O) \ge \frac{(j_{d/2-1})^d}{(d+2) 2^{d-2} \G(d/2+1)^2}.
\end{equation*}
hold true for all domains $\O$.
Here and everywhere below we denote by $j_\nu$ the first positive root
of the Bessel function $J_\nu$. 
\end{theorem}

\begin{cor}
\label{c12}
There is an absolute constant $C_1 > 0$ such that
\begin{equation*}
%\label{13} 
\Phi (d,1,\O) \ge C_1 \left(\frac{e}2\right)^d \qquad \forall \ \O.
\end{equation*}
\end{cor}

Theorem \ref{t11} implies also many inequalities for individual eigenvalues.
As an example, we give

\begin{cor}
\label{c13}
Let $d \ge 6$, $\O \subset \R^d$ be a bounded domain such that
the imbedding $W_2^1 (\O) \subset L_2 (\O)$ is compact.
Then
$$
\mu_4 < \la_2.
$$
\end{cor}

See also the table at the end of the proof of Corollary \ref{c13} below.

Note also the elegant paper by L.~Hatcher \cite{Hatcher}.
Answering the question posed in \cite{CMS} on the estimates of $\Phi(d,1,\O)$
in geometric terms, he showed that
$$
c_1(d) \,\frac{\mes_{d-1}(\dd\O)^d}{\mes_d (\O)^{d-1}}
\le \Phi(d,1,\O) \le 
c_2(d) \,\frac{\mes_{d-1}(\dd\O)^d}{\mes_d (\O)^{d-1}}
$$
for all convex domains $\O$.
The constants $c_1$ and $c_2$ depend only on $d$.
However, $\log c_1 (d)$ is of order of $d^2$ for large $d$
which is too large for us.

P.~Freitas and M.~Gama in \cite{FrG} showed that $\Psi (d,k,\O)$ 
grows as a power of $k$ as $k\to\infty$.
We are interested in another regime, 
namely what happens if $d \to \infty$.

Our second result is

\begin{theorem}
\label{t14}
There is an absolute constant $C_{conv} > 0$ such that
\begin{equation*} 
\Psi (d,k,\O) \ge C_{conv} \left(\frac{e}2\right)^d 
\end{equation*}
for all natural $k$ and for all bounded convex domains $\O$.
\end{theorem}

Theorem \ref{t14} together with Corollary \ref{c12} invite us to make the following

\begin{conj}
\label{conj15}
There exist two absolute positive constants $C_*$, $\al_*$ such that
$$
\Psi (d,k,\O) \ge C_* e^{\al_* d} \qquad \forall \ k \in \N, \quad \forall \ \O \subset \R^d.
$$
\end{conj}

It is natural to consider the examples of multi-dimensional cubes and balls.
It is shown in \cite{CMS} that a) in the case of a ball the function $\Phi (d,1,B)$ 
grows faster than any power of $d$ (Theorem 3 in \cite{CMS});
and b) in the case of a cube it grows exponentially,
$\al_1 d \le \log \Phi (d,1,Q) \le \al_2 d$ 
with some positive constants $\al_1, \al_2$ (formula (10) in \cite{CMS}).
We refine these results by explicitly computing 
the asymptotics of $\log \Phi (d,1)$ in these cases.

\begin{theorem}
\label{t16}
Let $B$ be a $d$-dimensional ball.
Then
$$
\log \Phi (d, 1, B) = \al_{ball} d + \hat \al d^{1/3} + O (\log d),\qquad d \to + \infty,
$$
Here
$$
\al_{ball} = \sqrt 2 \log (\sqrt 2 + 1) - \log 2 \approx 0,55,
$$
$\hat \al \approx 1,84$, see \eqref{410} below.
\end{theorem}

Therefore, if the Conjecture \ref{conj15} holds true
then $\al_* \le \al_{ball}$.

\begin{theorem}
\label{t17}
Let $Q$ be a $d$-dimensional cube.
Then

a) $\Psi (d,k,Q) \ge \Psi (d,1,Q)$ for all $k\in\N$.

b) The asymptotics
$$
\log \Phi (d, 1, Q) = \al_{cube} d - \frac12 \log d + O (1), \qquad d \to + \infty,
$$
holds true.
Moreover,
$$
\Phi (d, 1, Q) = \frac\be{\sqrt d} \,e^{\al_{cube} d} \left(1 + O \left(\frac1d\right)\right),
\qquad d \to + \infty,
$$
$\al_{cube} \approx 1,04$, $\be \approx 0,74$, see \eqref{consts} below.
\end{theorem}

For further references introduce the counting functions 
for the spectra of the Dirichlet and Neumann problems,
$$
N_{\cal D} (\O,\L) := \# \{ j: \la_j \le \L\}, \qquad
N_{\cal N} (\O,\L) := \# \{ j: \mu_j \le \L\} .
$$
In these terms
\begin{equation*}
%\label{14}
\Phi (d, k, \O) = N_{\cal N} (\O, \la_k).
\end{equation*}
If $\la_k$ is a simple eigenvalue of the Dirichlet problem then
$$
\Psi (d, k, \O) = N_{\cal N} (\O, \la_k) - N_{\cal D} (\O, \la_k).
$$
In general case,
\begin{equation}
\label{15}
\Psi (d, k, \O) \ge N_{\cal N} (\O, \la_k) - N_{\cal D} (\O, \la_k).
\end{equation}

In the next section we prove Theorem \ref{t11} and Corollaries \ref{c12} and \ref{c13}.
In \S 3 we prove Theorem \ref{t14}.
In \S 4 we prove Theorem \ref{t16},
and in \S 5 we prove Theorem \ref{t17}.

{\bf Acknowledgements.} The results were partially obtained during the programme
{\it Geometric spectral theory and applications}, supported by EPSRC grant EP/Z000580/1.
The author thanks the Isaac Newton Institute for Mathematical Sciences, Cambridge, 
for support and hospitality.
The author also thanks Sergei V. Ivanov for consultation on the Kakeya-P\'al problem,
see \S \ref{smw} below.

%%%%%%%%%%%%%%%%%%%%%%%%%%%%%%%%%%%%%%%%%%
\section{Bounds for $\Phi (d, 1, \O)$ and $\Phi (d, 2, \O)$ for arbitrary domains}

We use three following famous inequalities.

\begin{theorem}[Kr\" oger, \cite{Kr}]
\label{t21}
Let $\O \subset \R^d$ be a bounded domain such that
the embedding $W_2^1 (\O) \subset L_2 (\O)$ is compact.
Then the Neumann counting function satisfies the bound
$$
N_{\cal N} (\O,\la) \ge \frac2{d+2} \cdot \frac{\oo_d \mes_d \O}{(2\pi)^d}\ \la^{d/2}.
$$
\end{theorem}

Here and in what follows we denote by $\oo_d$ 
the volume of the unit ball in $\R^d$,
$$
\oo_d = \frac{\pi^{d/2}}{\G \left(\frac{d}2+1\right)}.
$$

\begin{theorem}[Faber-Krahn, \cite{Faber, Krahn}]
\label{t22}
Let $\O \subset \R^d$ be a bounded domain.
Then 
$$
\la_1 (\O) \ge \la_1 (B_R),
$$
where $B_R$ is a ball of the same volume, $\mes_d B_R = \mes_d \O$.
\end{theorem}

\begin{theorem}[Krahn-Szego, \cite{Kra}]
\label{t23}
Let $\O \subset \R^d$ be a bounded domain.
Then 
$$
\la_2 (\O) \ge \la_1 (B_{\tilde R}),
$$
where $B_{\tilde R}$ is a ball such that $\mes_d B_{\tilde R} = \frac12 \mes_d \O$.
\end{theorem}

{\it Proof of Theorem \ref{t11}.}
The first eigenfunction and the first eigenvalue 
of the Dirichlet problem in a ball of radius $R$ are well known:
\begin{equation}
\label{21}
\ph_1 (x) = |x|^{1-d/2} J_{\frac{d}2-1} \left(\frac{j_{\frac{d}2-1}|x|}R\right), 
\qquad
\la_1 (B_R) = \left(\frac{j_{\frac{d}2-1}}R\right)^2.
\end{equation}

The equalities $\mes_d B_R = \mes_d \O$ and $\mes_d B_{\tilde R} = \frac12 \mes_d \O$
imply
$$
R = \frac{(\mes_d \O)^{1/d}}{\oo_d^{1/d}}, \qquad
\tilde R = \frac{(\mes_d \O)^{1/d}}{(2 \oo_d)^{1/d}}.
$$
Therefore,
$$
\la_1 (\O) \ge \frac{j_{\frac{d}2-1}^2 \oo_d^{2/d}}{(\mes_d \O)^{2/d}},
\qquad 
\la_2 (\O) \ge \frac{j_{\frac{d}2-1}^2 2^{2/d} \oo_d^{2/d}}{(\mes_d \O)^{2/d}}.
$$
Now, Theorem \ref{t21} yields the bounds
\begin{eqnarray*}
\Phi (d, 1, \O) = N_{\cal N} (\O, \la_1) 
\ge \frac2{d+2} \cdot \frac{\oo_d \mes_d \O}{(2\pi)^d}\ \la_1 (\O)^{d/2} \\
\ge \frac{2 \oo_d^2 \left(j_{\frac{d}2-1}\right)^d}{(d+2) (2\pi)^d}
= \frac{\left(j_{\frac{d}2-1}\right)^d}{(d+2) 2^{d-1} \G(d/2+1)^2},
\end{eqnarray*}
and in the same manner,
\begin{eqnarray*}
\Phi (d, 2, \O) 
\ge \frac2{d+2} \cdot \frac{\oo_d \mes_d \O}{(2\pi)^d}\ \la_2 (\O)^{d/2}
\ge \frac{(j_{\frac{d}2-1})^d}{(d+2) 2^{d-2} \G(d/2+1)^2}.
\quad \qed
\end{eqnarray*}

{\it Proof of Corollary \ref{c12}.}
Taking the logarithm of the inequality \eqref{12} we obtain 
\begin{equation}
\label{22}
\log \Phi (d,1,\O) \ge d \log j_{\frac{d}2-1} - d \log 2 - 2 \log \G\left(\frac{d}2 + 1\right)
+ O (\log d), \quad d \to \infty.
\end{equation}
The asymptotics of the first zeros of the Bessel functions 
$$
j_\nu = \nu + |a_1| 2^{-1/3} \nu^{1/3} + O \left(\nu^{-1/3}\right), 
\qquad \nu \to + \infty,
$$
is well known, see for example \cite{QuWong}.
Here $a_1$ is the first zero of the Airy function $Ai$, $a_1 \approx - 2,34$.
Therefore,
\begin{equation}
\label{23}
j_{\frac{d}2-1} = \frac{d}2 + |a_1| 2^{-2/3} d^{1/3} + O(1), \quad d \to \infty,
\end{equation}
\begin{equation}
\label{24}
\log j_{\frac{d}2-1} = \log d - \log 2 + |a_1| 2^{1/3} d^{-2/3} + O(d^{-1}),
\quad d \to \infty.
\end{equation}
We use also Stirling's formula 
\begin{equation}
\label{25}
\log \G\left(\frac{d}2 + 1\right)
= \frac{d}2 \left(\log d - \log 2 - 1\right) + O(\log d), \quad d \to \infty.
\end{equation}
Substituting \eqref{24} and \eqref{25} into \eqref{22} we get
$$
\log \Phi (d,1,\O) \ge 
(1 - \log 2) d + |a_1| 2^{1/3} d^{1/3} + O(\log d)
\ge (1 - \log 2) d
$$
for sufficiently large $d$.
This yields the existence of such a constant $C_1 > 0$ that
$$
\Phi (d,1,\O) \ge C_1 \left(\frac{e}2\right)^d \qquad \qed
$$

For the proof of Corollary \ref{c13} we use

\begin{lemma}
\label{l21}
The function
$$
f_0 (x) = \log (1+x) - \frac{2x}3
$$
is positive on the interval $(0,1]$.
\end{lemma}

\begin{proof}
We have $f_0''(x) = \frac{-1}{(1+x)^2}$, thus, 
the function $f_0$ is concave.
Next, 
$$
f_0 (0) = 0, \qquad f_0(1) = \log 2 - \frac23 > 0 .
$$
Therefore,
$f_0(x) > 0$ if $0<x \le 1$.
\end{proof}

Denote the right hand side of \eqref{12} by
$$
f_1 (d) := \frac{(j_{d/2-1})^d}{(d+2) 2^{d-1} \G(d/2+1)^2}.
$$
From \cite{QuWong} we know that
$$
j_\nu \ge \nu + |a_1| 2^{-1/3} \nu^{1/3} ,
$$
and therefore
$$
j_{\nu-1} \ge (\nu - 1) \left(1  + |a_1| 2^{-1/3} \nu^{-2/3}\right).
$$ 
It is known also that
$$
\G (x+1) = x \G(x) < x^{x+1/2} e^{-x} \sqrt{2\pi} \,e^{\frac1{12x}},
$$
so,
$$
\G\left(\frac{d}2+1\right) < 
\sqrt{2\pi} \left(\frac{d}2\right)^{\frac{d+1}2} e^{-\frac{d}2 + \frac1{6d}}.
$$
Combining all together we obtain
$$
f_1(d) \ge \frac{(d-2)^d \left(1 + |a_1|2^{1/3} d^{-2/3}\right)^d e^d}
{\pi (d+2) 2^{d-1} d^{d+1} e^{\frac1{3d}}} =: f_2 (d).
$$

\begin{lemma}
\label{l22}
The function $f_2(d)$ is increasing for $d \ge 8$.
\end{lemma}

\begin{proof}
We have
\begin{eqnarray*}
\log f_2 (d) = d \log (d-2) + d \log \left(1 + |a_1|2^{1/3} d^{-2/3}\right) + d \\
- \log \pi - \log (d+2) - (d-1) \log 2 - (d+1) \log d - \frac1{3d},
\end{eqnarray*}
\begin{eqnarray*}
\left(\log f_2 (d)\right)' 
= \frac{d}{d-2} + \log (d-2) + \log \left(1 + |a_1|2^{1/3} d^{-2/3}\right)
- \frac{|a_1| 2^{4/3}}{3 \left(1 + |a_1|2^{1/3} d^{-2/3}\right) d^{2/3}} \\
+1 - \frac1{d+2} - \log 2 - \frac{d+1}d - \log d + \frac1{3d^2}.
\end{eqnarray*}
Furthermore,
$$
\frac{d}{d-2} - \frac1{d+2} - \frac{d+1}d = \frac{6d+4}{d(d^2-4)} > 0,
$$
$$
\log \left(1 + |a_1|2^{1/3} d^{-2/3}\right)
- \frac{|a_1| 2^{4/3}}{3 \left(d^{2/3} + |a_1|2^{1/3}\right)}
> \log \left(1 + |a_1|2^{1/3} d^{-2/3}\right)
- \frac{|a_1| 2^{4/3}}{3 d^{2/3}} > 0
$$ 
by virtue of Lemma \ref{l21} as we have $|a_1| 2^{1/3} d^{-2/3} \le 1$
if $d \ge \sqrt 2 |a_1|^{3/2} \approx 5,06$.
Thus,
$$
\left(\log f_2 (d)\right)' \ge
\log (d-2) + 1 - \log 2 - \log d = \log \left(\frac{(d-2) e}{2d}\right) > 0,
$$
whenever
$d > \frac{2e}{e-2} \approx 7,57$.
\end{proof}

{\it Proof of Corollary \ref{c13}.}
The preceding Lemma yields
\begin{equation}
\label{26}
f_1 (d) \ge f_2 (d) \ge f_2 (10) \approx 1,61 \qquad
\text{if} \quad d \ge 10.
\end{equation}
For any $\er \in (0,1)$ denote by $\tilde \er (\O)$ the number
$$
\tilde \er (\O) =  \frac{j_{\frac{d}2-1}^2 2^{2/d} \oo_d^{2/d}}{(\mes_d \O)^{2/d}}
\left(1 - (1-\er)^{2/d}\right) > 0.
$$
Then repeating the arguments from the proof of Theorem \ref{t11} for 
the counting function at the point $(\la_2 (\O) - \tilde \er(\O))$ one obtain
$$
N_{\cal N} (\O, \la_2 - \tilde \er (\O)) 
\ge \frac2{d+2} \cdot \frac{\oo_d \mes_d \O}{(2\pi)^d} 
\left(\frac{j_{\frac{d}2-1}^2 2^{2/d} \oo_d^{2/d}}
{(\mes_d \O)^{2/d}} - \tilde \er (\O)\right)^{\frac{d}2}
= 2 f_1 (d) (1 - \er).
$$
For sufficiently small $\er$ the inequality \eqref{26} implies
$$
N_{\cal N} (\O, \la_2 - \tilde \er) \ge 3,2 \quad \Rightarrow
\quad \mu_4 \le \la_2 - \tilde \er < \la_2
$$
whenever $d \ge 10$.
For small values of $d$ we calculate the function $f_1(d)$ explicitly:
$$
\begin{array}{|c|c|c|c|c|}
\hline
& & & &\\
d & j_{\frac{d}2-1} & f_1(d) & \text{inequality with}\ \la_1 &
\text{inequality with}\ \la_2 \\
\hline
& & & &\\
6 & 5,1356 & 1,99 & & \mu_4 < \la_2 \\
\hline
& & & &\\
7 & 5,7635 & 2,71 & & \mu_6 < \la_2\\
\hline
& & & &\\
8 & 6,3802 & 3,72 & \mu_4 < \la_1 & \mu_8 < \la_2\\
\hline
& & & &\\
9 & 6,9879 & 5,15 & \mu_6 < \la_1 & \mu_{11} < \la_2\\
\hline
\end{array}
$$
\qed

%%%%%%%%%%%%%%%%%%%%%%%%%%%%%%%%%%%%%%%%%%
\section{Convex domains}
In this section we prove Theorem \ref{t14}.

\subsection{Estimate for the difference $N_{\cal N} (\la) - N_{\cal D} (\la)$
via the exponentials}
The following Lemma is proved in \cite{FilSaf}.
We give a proof for the sake of completeness.

\begin{lemma}
\label{l31}
Let $\la > 0$.
Introduce the subspace
\begin{equation}
\label{30}
H_{\cal D} (\la) = \operatorname{Lin} \left\{\ph_1, \dots, \ph_k\right\}_{\la_k \le \la},
\end{equation}
where $\ph_j$ is the Dirichlet eigenfunction corresponding to $\la_j$.
Let $F \subset W_2^1 (\O)$ be a subspace such that
$$
F \cap H_{\cal D} = \{0\}, \qquad
- \D v = \la v, \qquad
\|\n v\|_{L_2(\O)}^2 \le \la \|v\|_{L_2(\O)}^2, \quad \forall \ v \in F.
$$
Then
$$
N_{\cal N} (\O, \la) - N_{\cal D} (\O, \la) \ge \dim F.
$$
\end{lemma}

\begin{proof}
Let us consider the subspace $H_{\cal D} + F \subset W_2^1 (\O)$.
By assumptions,
$$
\dim \left(H_{\cal D} + F\right) = N_{\cal D} (\O, \la) + \dim F.
$$
Pick a function $w$ from this subspace,
$$
w = u + v, \quad u \in H_{\cal D}, \quad v \in F.
$$
Then we have
\begin{eqnarray*}
\int_\O |\n w|^2 dx = 
\int_\O \left(|\nabla u|^2 + |\nabla v|^2 +
2\, \re \langle \nabla u, \nabla v\rangle \right) dx, \\
\int_\O |\n u|^2 dx \le \la \int_\O |u|^2 dx, \quad
\int_\O |\n v|^2 dx \le \la \int_\O |v|^2 dx, \\
\int_\O \langle \nabla u, \nabla v\rangle 
= - \int_\O u \D \overline v \, dx = \la \int_\O u \overline v \, dx.
\end{eqnarray*}
Therefore,
$$
\int_\O |\n w|^2 dx
\le \la \int_\O\left(|u|^2 + |v|^2 + 2 \re (u \overline v)\right) dx
= \la \int_\O |w|^2 dx.
$$
Now, the variational principle implies
$$
N_{\cal N} (\O, \la) \ge \dim \left(H_{\cal D} + F\right).
\qquad \qedhere
$$
\end{proof}

Taking into account this Lemma and the inequality \eqref{15},
our goal now is to construct a such subspace $F$ for $\la \ge \la_1 (\O)$.

\subsection{Minimal width}
\label{smw}
Let $\O \subset \R^d$ be a bounded convex domain.
The distance between two supporting hyperplanes of $\O$
orthogonal to a vector $\vec e \in \R^d$ is called the width
of $\O$ in the direction $\vec e$.
We denote by $w_{\min} (\O)$ the smallest width of $\O$ 
over all directions.

\begin{lemma}
\label{l32}
Let $\O \subset \R^d$ be a bounded convex domain.
For each $k = 1, \dots, d$, there exists a $k$-dimensional plane $H_k$,
$\dim H_k = k$, such that
\begin{equation}
\label{31}
\mes_k (P_k) \ge \frac{w_{\min} (\O)^k}{k!},
\qquad \text{where}\quad P_k = \O \cap H_k.
\end{equation}
\end{lemma}

\begin{proof}
We prove this Lemma by induction in $k$.

Clearly, one can find two points $x_1^{(1)}, x_2^{(1)} \in \overline \O$
such that $|x_1^{(1)} - x_2^{(1)}| \ge w_{\min} (\O)$.
Now, take the line passing through these two points as $H_1$.
Thus, the claim in the case $k=1$ is proven.

Assume that $H_k$ satisfies \eqref{31}.
Choose a unit vector $\vec e_{k+1}$ orthogonal to the plane $H_k$.
We take the linear hull of $H_k$ and $\vec e_{k+1}$ as the next plane $H_{k+1}$,
$$
H_{k+1} = \left\{x = y + t \vec e_{k+1},\ y \in H_k, t \in \R\right\}.
$$
By definition of the minimal width of $\O$,
there are two points $x_1^{(k+1)}, x_2^{(k+1)} \in \overline \O$ such that
$$
x_1^{(k+1)} = y_1 + t_1 \vec e_{k+1}, \quad
x_2^{(k+1)} = y_2 - t_2 \vec e_{k+1}, \quad
y_1, y_2 \in H_k, \quad t_1, t_2 \ge 0, \quad t_1 + t_2 \ge w_{\min} (\O).
$$
Then the convex set $\overline{P_{k+1}} = \overline \O \cap H_{k+1}$
contains the convex hull of the set $P_k$ and of both points $x_1^{(k+1)}$, $x_2^{(k+1)}$.
The measure of this convex hull is at least
$$
\frac{t_1 \mes_k P_k}{k+1} + \frac{t_2 \mes_k P_k}{k+1}. 
$$
Now, due to \eqref{31}
$$
\mes_{k+1} P_{k+1} \ge \frac{(t_1+ t_2) \mes_k P_k}{k+1} 
\ge \frac{w_{\min} (\O)^{k+1}}{(k+1)!},
$$
which is \eqref{31} for the index $k+1$.
\end{proof}

Choosing $k=d$ in \eqref{31} we obtain

\begin{theorem}
\label{t33}
Let $\O \subset \R^d$ be a bounded convex domain.
Put
\begin{equation}
\label{32}
w \equiv w(\O) := \left(d! \mes_d \O\right)^{1/d}.
\end{equation}
Then
\begin{equation}
\label{33}
w_{\min} (\O) \le w (\O).
\end{equation}
\end{theorem}

\begin{rem}
We provided the proof of Theorem \ref{t33} for the sake of completeness.
The question about minimizing the volume of convex domain of given minimal width in $\R^d$
is known as Kakeya-P\'al problem.
The answer is known \cite{Pal} for the case $d=2$ only:
the optimal domain is a regular triangle of height $w$.
This fact yields the following improvement of the estimate \eqref{33}:
$$
w_{\min} (\O)^d \le \frac{\sqrt 3}2\, d! \mes_d \O.
$$
It was done in \cite{Firey}.
Another improvement is contained in \cite[the proof of Theorem 6.2]{Bezdek}.
\end{rem}

In Theorem \ref{t14} we assume that the domain $\O$ is convex.
The only advantage of it we use is the fact that $\O$ is contained in a layer
of width $w(\O)$.

\begin{theorem}
\label{t335}
Assume that a bounded domain $\O \subset \R^d$ is contained in a layer
of width $w(\O)$ defined in \eqref{32}.
Then 
\begin{equation*} 
\Psi (d,k,\O) \ge C_{conv} \left(\frac{e}2\right)^d 
\end{equation*}
for all $k$.
Here $C_{conv}>0$ is an absolute constant.
\end{theorem}

Clearly, Theorem \ref{t14} follows from Theorem \ref{t33} and Theorem \ref{t335}.

\subsection{Functions $u_\eta$}
In the sequel in this section we assume
$$
x = (x_1, y) \in \R^d, \qquad x_1 \in \R, \quad y \in \R^{d-1},
$$
$$
\O \subset\left\{x = (x_1, y) : |x_1| \le \frac{w}2\right\},
$$
where the number $w = w(\O)$ is defined by the formula \eqref{32}.

\begin{lemma}
\label{l36}
Let $\de > 0$.
Then
$$
\int_\O \cos (2 x_1 \de) dx \ge \left(1 - \frac{w^2\de^2}2\right) \mes_d \O.
$$
\end{lemma}

\begin{proof}
The inequality
$$
0 \le 1 - \cos z \le \frac{z^2}2 \qquad \forall \ z \in \R
$$
implies the bound
$$
\left| \mes_d \O - \int_\O \cos (2 x_1 \de) dx\right|
\le \int_\O \left|1-\cos (2 x_1 \de)\right| dx
\le \int_\O 2 x_1^2 \de^2 dx
\le \frac{w^2\de^2}2\, \mes_d \O.
\quad \qedhere
$$
\end{proof}

Let
$$
\la \ge \de^2, \qquad \eta \in \R^{d-1}, \qquad |\eta|^2 + \de^2 = \la.
$$
Introduce the function
\begin{equation}
\label{34}
u_\eta (x) = \cos (x_1 \de) \,e^{i\<y,\eta\>},
\qquad \text{where} \quad x = (x_1, y).
\end{equation}
It is clear that
\begin{equation}
\label{35}
- \D u_\eta = \la u_\eta.
\end{equation}

\begin{lemma}
\label{l37}
Let
$$
\eta, \ze \in \R^{d-1}, \qquad
|\eta| = |\ze| = \sqrt{\la - \de^2}.
$$
Then
$$
\la \int_\O u_\eta\, \overline{u_\ze}\, dx - \int_\O \<\n u_\eta, \n u_\ze\> dx
= \int_\O \left(\de^2 \cos(2x_1\de) + \frac{|\eta-\ze|^2}2 \cos^2 (x_1\de)\right)
e^{i\<y, \eta-\ze\>} dx.
$$
\end{lemma}

\begin{proof}
We have
$$
\int_\O u_\eta\, \overline{u_\ze}\, dx
= \int_\O \cos^2 (x_1\de) e^{i\<y, \eta-\ze\>} dx,
$$
$$
\int_\O \<\n u_\eta, \n u_\ze\> dx
= \int_\O \left(\de^2 \sin^2 (x_1 \de) + \<\eta, \ze\> \cos^2 (x_1 \de)\right) e^{i\<y, \eta-\ze\>} dx.
$$
Taking into account the identity
$$
\frac{|\eta-\ze|^2}2 = \la - \de^2 - \<\eta,\ze\>,
$$
we obtain the result.
\end{proof}

Lemma \ref{l36} and Lemma \ref{l37} imply the following

\begin{cor}
\label{c38}
Let $\eta \in \R^{d-1}$, $|\eta| = \sqrt{\la - \de^2}$.
Then
$$
\la \int_\O |u_\eta|^2 dx - \int_\O |\n u_\eta|^2 dx 
\ge \de^2 \left(1 - \frac{w^2\de^2}2\right) \mes_d \O.
$$
\end{cor}

In the sequel we take
$$
\de = \frac1{w(\O)}.
$$
Then

\begin{equation}
\label{36}
\la \int_\O |u_\eta|^2 dx - \int_\O |\n u_\eta|^2 dx 
\ge \frac{\mes_d \O}{2 w(\O)^2}.
\end{equation}

\begin{rem}
In the sequel we consider 
$$
\la \ge \la_1 (\O) \ge \frac{j_{\frac{d}2-1}^2 \oo_d^{2/d}}{(\mes_d \O)^{2/d}},
\qquad \text{and} \qquad \de^2 = \frac1{\left(d! \mes_d\O\right)^{2/d}}.
$$
Note that for each $d \ge 2$ we have
$j_{\frac{d}2-1} \ge j_0 > 1$, and
$$
\oo_d = \frac{\pi^{d/2}}{\G \left(\frac{d}2+1\right)} 
> \frac1{\G \left(\frac{d}2+1\right)} > \frac1{d!},
$$ 
so the condition $\la \ge \de^2$ is fulfilled.
\end{rem}

\subsection{Matrix $B(H)$}
We consider the set of $N$ points on the sphere $S_{\sqrt{\la-\de^2}}^{d-2}$,
$$
H := \left\{\eta_1, \dots, \eta_N\right\} \in \left(S_{\sqrt{\la-\de^2}}^{d-2}\right)^N.
$$
The number $N$ will be chosen later.
For each such set $H$ we define the linear space
$$
F (H) = \operatorname{Lin} \left\{u_{\eta_1}, \dots, u_{\eta_N}\right\},
$$
the functions $u_\eta$ here are defined by the formula \eqref{34}.
Clearly,
$$
F (H) \subset W_2^1 (\O).
$$
If all points $\eta_1, \dots, \eta_N$ are different then $\dim F(H) = N$.
Denote by $\Sigma^{(1)} \subset \left(S_{\sqrt{\la-\de^2}}^{d-2}\right)^N$
the set of such sets $H$ for which
\begin{equation}
\label{375}
F (H) \cap H_{\cal D} (\la) = \{0\},
\end{equation}
the space $H_{\cal D} (\la)$ is defined by \eqref{30}.

\begin{lemma}
\label{l39}
The set $\Sigma^{(1)}$ is everywhere dense in $\left(S_{\sqrt{\la-\de^2}}^{d-2}\right)^N$.
\end{lemma}

\begin{proof}
Pick an arbitrary set
$$
\left\{\eta_1, \dots, \eta_N\right\} \in \left(S_{\sqrt{\la-\de^2}}^{d-2}\right)^N.
$$
All functions $\{u_\eta\}_{\eta \in S_{\sqrt{\la-\de^2}}^{d-2}}$ are linearly independent.
Therefore in an arbitrarily small neighborhood of the point $\eta_1$
one can find a point $\tilde \eta_1$ such that
$u_{\tilde \eta_1} \notin H_{\cal D} (\la)$.
In any neighborhood of the point $\eta_2$ one can find a point $\tilde \eta_2$ such that
$$
u_{\tilde \eta_2} \notin \operatorname{Lin} \left\{u_{\tilde \eta_1}, \ph_1, \dots, \ph_k\right\}.
$$
In any neighborhood of the point $\eta_3$ one can find a point $\tilde \eta_3$ such that
$$
u_{\tilde \eta_3} \notin 
\operatorname{Lin} \left\{u_{\tilde \eta_1}, u_{\tilde \eta_2}, \ph_1, \dots, \ph_k\right\},
$$
and so on.
For the final set $\{\tilde\eta_1, \dots, \tilde\eta_N\}$ all the functions
$u_{\tilde \eta_1}, \dots, u_{\tilde \eta_N}, \ph_1, \dots, \ph_k$
are linearly independent, and therefore,
$$
F \left(\left\{\tilde\eta_1, \dots, \tilde\eta_N\right\}\right) \cap H_{\cal D} (\la) = \{0\},
$$
and 
$\left\{\tilde\eta_1, \dots, \tilde\eta_N\right\} \in \Sigma^{(1)}$.
\end{proof}

If 
$$
f = \sum_{j=1}^N c_j u_{\eta_j} \in F (H), \qquad c_j \in \C,
$$
then by \eqref{35}
$$
- \D f(x) = \la f(x),
$$
and
$$
\la \int_\O |f(x)|^2 dx - \int_\O |\n f(x)|^2 dx = (B \vec c, \vec c).
$$
Due to Lemma \ref{l37} the entries of the $(N\times N)$-matrix $B$ equal to
\begin{eqnarray*}
B_{kj} =
\la \int_\O u_{\eta_j}\, \overline{u_{\eta_k}}\, dx 
- \int_\O \<\n u_{\eta_j}, \n u_{\eta_k}\> dx \\
= \int_\O \left(\de^2 \cos(2x_1\de) + \frac{|\eta_j-\eta_k|^2}2 \cos^2 (x_1\de)\right)
e^{i\<y, \eta_j-\eta_k\>} dx.
\end{eqnarray*}
Represent the matrix $B$ as a sum of diagonal and off-diagonal parts:
$$
B = B_0 + B_1, \qquad B_0 = \operatorname{diag} (B_{jj}).
$$
By virtue of Corollary \ref{c38} and formula \eqref{36} we have
\begin{equation}
\label{38}
B_0 \ge \frac{\mes_d \O}{2 w(\O)^2}\, I_{N\times N}.
\end{equation}
Let us estimate the Hilbert-Schmidt norm of the matrix $B_1$.
We have
\begin{eqnarray*}
|B_{kj}|^2 
= \left|\de^2 \int_\O \cos(2x_1\de) e^{i\<y, \eta_j-\eta_k\>} dx
+ \frac{|\eta_j-\eta_k|^2}2 \int_\O \cos^2 (x_1\de) e^{i\<y, \eta_j-\eta_k\>} dx\right|^2 \\
\le \left(\de^4 + \frac{|\eta_j-\eta_k|^4}4\right)
\left(G_1 (\eta_j-\eta_k) + G_2 (\eta_j-\eta_k)\right),
\end{eqnarray*}
where we used the notation
\begin{equation}
\label{39}
G_1 (\te) = \left|\int_\O \cos(2x_1\de) e^{i\<y,\te\>} dx\right|^2, \quad
G_2 (\te) = \left|\int_\O \cos^2 (x_1\de) e^{i\<y,\te\>} dx\right|^2,
\end{equation}
and the inequality $|ab+cd|^2 \le (|a|^2+|c|^2) (|b|^2 + |d|^2)$.
Therefore,
$$
\|B_1\|_{S_2}^2 = \sum_{j\neq k} |B_{kj}|^2
\le \sum_{j \neq k} \left(\de^4 + \frac{|\eta_j-\eta_k|^4}4\right)
\left(G_1 (\eta_j-\eta_k) + G_2 (\eta_j-\eta_k)\right).
$$
Now, integrate this inequality over all $H \in \left(S_{\sqrt{\la-\de^2}}^{d-2}\right)^N$:
\begin{eqnarray}
\nonumber
\int_{S_{\sqrt{\la-\de^2}}^{d-2}} \dots \int_{S_{\sqrt{\la-\de^2}}^{d-2}} 
\|B_1 (H)\|_{S_2}^2 dS (\eta_1) \dots dS (\eta_N) \\
\label{310}
\le N (N-1) \left(\mes_{d-2} S_{\sqrt{\la-\de^2}}^{d-2}\right)^{N-2} \times \\
\times \int_{S_{\sqrt{\la-\de^2}}^{d-2}} \int_{S_{\sqrt{\la-\de^2}}^{d-2}}
\left(\de^4 + \frac{|\eta-\ze|^4}4\right)
\left(G_1 (\eta-\ze) + G_2 (\eta-\ze)\right) dS (\eta) dS(\ze).
\nonumber
\end{eqnarray}

\subsection{Estimates of integrals}
The following identity is proved in \cite[Lemma 5.2]{FilSaf}.

\begin{lemma}
Let $m \ge 2$, $R>0$, $f \in C(\R^n)$.
Then
\begin{eqnarray*}
\int_{S_R^{m-1}} \int_{S_R^{m-1}} f (\eta-\ze) dS (\eta) dS (\ze) 
= (m-1)\, \oo_{m-1} R^2 \int_{B_{2R}^m} f(\te) \left(R^2 - \frac{\te^2}4\right)^{\frac{m-3}2}
\frac{d\te}{|\te|} .
\end{eqnarray*}
\end{lemma}

This Lemma yields the equality
\begin{eqnarray}
\label{311}
\int_{S_{\sqrt{\la-\de^2}}^{d-2}} \int_{S_{\sqrt{\la-\de^2}}^{d-2}}
\left(\de^4 + \frac{|\eta-\ze|^4}4\right)
\left(G_1 (\eta-\ze) + G_2 (\eta-\ze)\right) dS (\eta) dS(\ze) \\
= (d-2) \oo_{d-2} (\la - \de^2) 
\int_{B_{2\sqrt{\la-\de^2}}^{d-1}} \left(\de^4 + \frac{|\te|^4}4\right)
\left(G_1(\te) + G_2(\te)\right) \left(\la - \de^2 - \frac{\te^2}4\right)^{\frac{d-4}2} \frac{d\te}{|\te|} .
\nonumber
\end{eqnarray}

\begin{lemma}
\label{l311}
The estimates
$$
\int_{\R^{d-1}} \left|\int_\O \cos(2x_1\de) e^{i\<y,\te\>} dx\right|^2 d\te
\le (2\pi)^{d-1} w(\O) \mes_d \O,
$$
$$
\int_{\R^{d-1}} \left|\int_\O \cos^2 (x_1\de) e^{i\<y,\te\>} dx\right|^2 d\te
\le (2\pi)^{d-1} w(\O) \mes_d \O
$$
hold true.
\end{lemma}

\begin{proof}
Given $x_1$ denote by $\O' (x_1)$ a cross-section of the domain $\O$ 
with the corresponding hyperplane,
$$
\O' (x_1) = \left\{y \in \R^{d-1} : (x_1, y) \in \O\right\}.
$$
Then
$$
\left\|\int_{\O'(x_1)} e^{i\<y,\te\>} dy\right\|_{L_2(\R^{d-1})}^2 
= \int_{\R^{d-1}} \left|\int_{\O'(x_1)} e^{i\<y,\te\>} dy\right|^2 d\te
= (2\pi)^{d-1} \mes_{d-1} \O' (x_1),
$$
as the Fourier transform preserves $L_2$-norm.
Therefore,
\begin{eqnarray*}
\left\|\int_\O \cos (2x_1\de) e^{i\<y,\te\>} dx\right\|_{L_2(\R^{d-1})}
= \left\|\int_{-w/2}^{w/2} \cos (2x_1\de) dx_1 
\int_{\O'(x_1)} e^{i\<y,\te\>} dy\right\|_{L_2(\R^{d-1})} \\
\le \int_{-w/2}^{w/2} dx_1 \left\|\int_{\O'(x_1)} e^{i\<y,\te\>} dy\right\|_{L_2(\R^{d-1})} 
= \int_{-w/2}^{w/2} dx_1 (2\pi)^{\frac{d-1}2} \left(\mes_{d-1} \O' (x_1)\right)^{1/2} \\
\le (2\pi)^{\frac{d-1}2} \sqrt{w} \left(\int_{-w/2}^{w/2} dx_1 \mes_{d-1} \O' (x_1)\right)^{1/2}
= (2\pi)^{\frac{d-1}2} \sqrt{w} \left(\mes_d \O\right)^{1/2}.
\end{eqnarray*}
The proof of the second inequality is similar.
\end{proof}

\begin{lemma}
\label{l312}
Let $m \ge 2$, $a,b \ge 0$, $G \in C(\R^m)$,
$$
0 \le G(\te) \le a \quad \forall \ \te \in \R^m, 
\qquad \int_{\R^m} G(\te) \, d\te \le b.
$$
Then
$$
\int_{\R^m} \frac{G(\te)\, d\te}{|\te|} \le 
\frac{m}{m-1} \left(\oo_m a b^{m-1}\right)^{1/m}.
$$
\end{lemma}

\begin{proof}
Let us consider the function
$$
h(r) = \int_{|\te| \ge r} \frac{G(\te)\, d\te}{|\te|}.
$$
Clearly, the integral $\int_{\R^m} \frac{G(\te)\, d\te}{|\te|}$ converges,
thus, the function $h$ is well defined, non-negative and continuous on $[0,\infty)$.
Moreover,
$$
r h(r) \le \int_{|\te| \ge r} G(\te)\, d\te, \quad \text{so} \quad
r h(r) \mathop{\longrightarrow}\limits_{r \to \infty} 0.
$$
Next,
\begin{equation}
\label{312}
h'(r) = - \frac1r \int_{|\te|=r} G(\te) \, dS(\te),
\end{equation}
therefore,
$$
\int_{\R^m} G(\te) \, d\te = \int_0^\infty dr \int_{|\te|=r} G(\te) \, dS(\te)
= - \int_0^\infty r h'(r)\, dr = \int_0^\infty h(r) \, dr,
$$
and
\begin{equation}
\label{313}
\int_0^\infty h(r) \, dr \le b.
\end{equation}
Furthermore, \eqref{312} implies
$$
h'(r) \ge - m \oo_m a r^{m-2},
$$
so,
$$
h(r) = h(0) + \int_0^r h'(t)\,dt
\ge h(0) - m \oo_m a \int_0^r t^{m-2} dt 
= h(0) - \frac{m \oo_m a}{m-1} \, r^{m-1}.
$$
Denote
$$
r_0 = \left(\frac{(m-1) h(0)}{m \oo_m a}\right)^{\frac1{m-1}}.
$$
Then by virtue of \eqref{313}
\begin{eqnarray*}
b \ge \int_0^\infty h(r)\, dr 
\ge \int_0^{r_0} \left(h(0) - \frac{m \oo_m a}{m-1} \, r^{m-1}\right) dr \\
= h(0) r_0 - \frac{\oo_m a r_0^m}{m-1}
= \frac{m-1}m \, r_0 h(0)
= \left(\frac{(m-1) h(0)}m\right)^{\frac{m}{m-1}} \frac1{(\oo_m a)^{\frac1{m-1}}},
\end{eqnarray*}
and therefore,
$$
\left(\oo_m a b^{m-1}\right)^{1/m} \ge \frac{m-1}m \, h(0).
\qquad \qedhere
$$
\end{proof}

\begin{rem}
The constant in this estimate is sharp.
Indeed, if we consider the function
$$
G(\te) = \begin{cases}
1, \quad &|\te| < 1, \\
0, \quad &|\te| \ge 1,
\end{cases}
$$
then
$$
a = 1, \quad b = \oo_m \quad \text{and} \quad
\int_{\R^m} \frac{G(\te)\, d\te}{|\te|} = \frac{m \oo_m}{m-1}.
$$
\end{rem}

\begin{cor}
\label{c314}
Functions $G_1$, $G_2$ defined in \eqref{39} satisfy the estimates
$$
\int_{\R^{d-1}} \frac{G_j(\te)\, d\te}{|\te|} \le 
\frac{(d-1) (2\pi)^{d-2}}{d-2} 
\left(\oo_{d-1} w(\O)^{d-2} \left(\mes_d\O\right)^d\right)^{\frac1{d-1}},
\quad j = 1, 2.
$$
\end{cor}

\begin{proof}
It is clear by definition that
$$
0 \le G_j (\te) \le (\mes_d\O)^2, \quad j = 1, 2.
$$
So, the functions $G_1$, $G_2$ satisfy the assumptions of Lemma \ref{l312}
with 
$$
m=d-1, \quad a = (\mes_d\O)^2 \quad \text{and}
\quad b =  (2\pi)^{d-1} w(\O) \mes_d \O
$$ 
due to Lemma \ref{l311}.
\end{proof}

\begin{lemma}
\label{l315}
Let $d\ge 5$.
Then
$$
x^2 (1-x)^{\frac{d-4}2} \le \frac{16 (d-4)^{\frac{d-4}2}}{d^{\frac{d}2}},
\qquad x \in [0,1].
$$
\end{lemma}

\begin{proof}
The function
$g(x): = x^2 (1-x)^{\frac{d-4}2}$
vanishes at the endpoints of the interval,
$g(0) = g(1) = 0$,
and it is positive inside.
The derivative
$$
g'(x) = 2x (1-x)^{\frac{d-4}2} - \frac{d-4}2\, x^2 (1-x)^{\frac{d-6}2}
$$
has a single root $x = \frac4d$.
Therefore,
$$
g(x) \le g\left(\frac4d\right) = \frac{16 (d-4)^{\frac{d-4}2}}{d^{\frac{d}2}}.
\qquad \qedhere
$$
\end{proof}

\begin{cor}
\label{c316}
If $\de^2 + \frac{\te^2}4 \le \la$ then
$$
\left(\de^4 + \frac{|\te|^4}4\right) \left(\la - \de^2 - \frac{\te^2}4\right)^{\frac{d-4}2}
\le \frac{16 (d-4)^{\frac{d-4}2}}{d^{\frac{d}2}}\, \la^{\frac{d}2}.
$$
\end{cor}

\begin{proof}
Clearly,
$$
\left(\de^4 + \frac{|\te|^4}4\right) \left(\la - \de^2 - \frac{\te^2}4\right)^{\frac{d-4}2}
\le \left(\de^2 + \frac{\te^2}4\right)^2 \left(\la - \de^2 - \frac{\te^2}4\right)^{\frac{d-4}2}.
$$
Applying the preceding Lemma with 
$x = \frac{\de^2 + \te^2/4}\la$ to the right hand side we get the claim.
\end{proof}

Now, we are ready to estimate the integral in the right hand side of \eqref{311}:
\begin{eqnarray}
\nonumber
\int_{B_{2\sqrt{\la-\de^2}}^{d-1}} \left(\de^4 + \frac{|\te|^4}4\right)
\left(G_1(\te) + G_2(\te)\right) \left(\la - \de^2 - \frac{\te^2}4\right)^{\frac{d-4}2} \frac{d\te}{|\te|} \\
\label{314}
\le \frac{16 (d-4)^{\frac{d-4}2}}{d^{\frac{d}2}}\, \la^{\frac{d}2}
\int_{\R^{d-1}} \left(G_1(\te) + G_2(\te)\right) \frac{d\te}{|\te|} \\
\le \frac{32 (d-4)^{\frac{d-4}2} \la^{\frac{d}2} (d-1) (2\pi)^{d-2}}{d^{\frac{d}2} (d-2)}
\left(\oo_{d-1} w(\O)^{d-2} \left(\mes_d\O\right)^d\right)^{\frac1{d-1}}.
\nonumber
\end{eqnarray}
Here we used Corollary \ref{c316} on the first step and Corollary \ref{c314} on the second step.
Relations \eqref{310}, \eqref{311} and \eqref{314} yield 
\begin{eqnarray}
\nonumber
\int_{S_{\sqrt{\la-\de^2}}^{d-2}} \dots \int_{S_{\sqrt{\la-\de^2}}^{d-2}} 
\|B_1 (H)\|_{S_2}^2 dS (\eta_1) \dots dS (\eta_N) 
\le N (N-1) \left(\mes_{d-2} S_{\sqrt{\la-\de^2}}^{d-2}\right)^{N-2} \times \\
\times 32 (d-4)^{\frac{d-4}2} (d-1) d^{-\frac{d}2} \oo_{d-2} (2\pi)^{d-2} (\la-\de^2) \la^{\frac{d}2}
\left(\oo_{d-1} w(\O)^{d-2} \left(\mes_d\O\right)^d\right)^{\frac1{d-1}}.
\label{315}
\end{eqnarray}
The function $\|B_1 (H)\|_{S_2}^2$ is a continuous function of 
$H \in \left(S_{\sqrt{\la-\de^2}}^{d-2}\right)^N$.
Therefore, there is a non-empty open subset 
$$
\Sigma^{(2)} \subset \left(S_{\sqrt{\la-\de^2}}^{d-2}\right)^N
$$
where this function does not exceed the right hand side of \eqref{315}
divided by the volume $\left(\mes_{d-2} S_{\sqrt{\la-\de^2}}^{d-2}\right)^N$.
Taking into account that
$$
\mes_{d-2} S_{\sqrt{\la-\de^2}}^{d-2}
= (d-1) \oo_{d-1} (\la-\de^2)^{\frac{d-2}2},
$$
we obtain for $H \in \Sigma^{(2)}$
\begin{eqnarray}
\nonumber
\|B_1 (H)\|_{S_2}^2 
\le \frac{N (N-1)}{\left(\mes_{d-2} S_{\sqrt{\la-\de^2}}^{d-2}\right)^2} \times \\
\label{317}
\times 32 (d-4)^{\frac{d-4}2} (d-1) d^{-\frac{d}2} \oo_{d-2} (2\pi)^{d-2} (\la-\de^2) \la^{\frac{d}2}
\left(\oo_{d-1} w(\O)^{d-2} \left(\mes_d\O\right)^d\right)^{\frac1{d-1}} \\
= \frac{32 N (N-1) (d-4)^{\frac{d-4}2} \oo_{d-2} (2\pi)^{d-2} \la^{\frac{d}2}}
{(d-1) d^{\frac{d}2} \oo_{d-1}^2 (\la-\de^2)^{d-3}}  
\left(\oo_{d-1} w(\O)^{d-2} \left(\mes_d\O\right)^d\right)^{\frac1{d-1}} .
\nonumber
\end{eqnarray}
The set $\Sigma^{(2)}$ is open, and by Lemma \ref{l39} 
the set $\Sigma^{(1)}$ is everywhere dense.
Therefore, 
$$
\Sigma^{(1)} \cap \Sigma^{(2)} \neq \emptyset.
$$

We proved the following

\begin{lemma}
\label{l317}
For any natural $N$ there is $H = H^{(N)} \in \left(S_{\sqrt{\la-\de^2}}^{d-2}\right)^N$
such that the relations
\eqref{375} and \eqref{317} are fulfilled.
\end{lemma}

\subsection{Positive spectra of matrices of special kind}
Given a Hermitian matrix $B \in \mat (\C, N \times N)$
denote by $N_+ (B)$ the number of its positive eigenvalues,
$N_+(B) = \# \left(\spec (B) \cap (0, \infty)\right)$.
The number $N_+ (B)$ coincides with the maximal dimension
of a subspace of $\C^N$ where the quadratic form of the matrix $B$ is positive.

\begin{lemma}
\label{l318}
Let 
$$
B = B^* \in \mat (\C, N \times N), \qquad B = B_0 + B_1,
$$
where $B_0$ is a diagonal part of $B$ and $B_1$ is an off-diagonal part of $B$.
Assume that
\begin{equation}
\label{318}
B_0 \ge a I_{N \times N} \qquad \text{and} \qquad
\|B_1\|_{S_2}^2 \le b (N^2-N),
\end{equation}
with some positive $a$ and $b$.
Then
$$
N_+ (B) \ge N - \frac{b(N^2-N)}{a^2}.
$$
\end{lemma}

\begin{proof}
The second condition \eqref{318} implies that
$$
\# \left(\spec (B_1) \cap (-\infty, -a]\right) \le \frac{b(N^2-N)}{a^2}.
$$
Therefore,
$$
\# \left(\spec (B) \cap (0, \infty)\right)
\ge N - \frac{b(N^2-N)}{a^2}.
\quad \qedhere
$$
\end{proof}

\begin{lemma}
\label{l319}
Fix $a, b > 0$. 
Assume that there is a sequence of matrices $B^{(N)}$ 
satisfying the conditions of Lemma \ref{l318} for each $N \in \N$.
Then one can find such a number $N$ that
$$
N_+ (B^{(N)}) \ge \frac{a^2}{4b} + \frac12.
$$
\end{lemma}

\begin{proof}
We take as $N$ a nearest integer to the number $\frac{a^2 + b}{2b}$,
so
$$
N = \frac{a^2 + b}{2b} + \er, \qquad |\er| \le \frac12.
$$
Then we use Lemma \ref{l318}, and we have
\begin{eqnarray*}
N - \frac{b(N^2-N)}{a^2} = \frac{(a^2+b)N}{a^2} - \frac{bN^2}{a^2}
= \frac{a^2+b}{a^2} \left(\frac{a^2 + b}{2b} + \er\right) 
- \frac{b}{a^2} \left(\frac{a^2 + b}{2b} + \er\right)^2 \\
= \frac{(a^2+b)^2}{4a^2b} - \frac{b\er^2}{a^2}
\ge \frac{a^4 + 2a^2b + b^2}{4a^2b} - \frac{b}{4a^2} 
= \frac{a^2}{4b} + \frac12. \quad \qedhere
\end{eqnarray*}
\end{proof}

\subsection{Proof of Theorem \ref{t14}}
Recall that for each natural $N$ Lemma \ref{l317} guarantees the existence of a set 
$H^{(N)} \in \left(S_{\sqrt{\la-\de^2}}^{d-2}\right)^N$
such that the corresponding space $F(H^{(N)})$ possesses the following properties:
\begin{itemize}
\item $F(H^{(N)}) \subset W_2^1 (\O)$;
\item $F (H) \cap H_{\cal D} (\la) = \{0\}$;
\item $\dim F(H^{(N)}) = N$;
\item $- \D f = \la f$ for all $f \in F(H^{(N)})$;
\item if 
$f = \sum_{j=1}^N c_j u_{\eta_j} \in F(H^{(N)})$ then
\begin{equation}
\label{319}
\la \int_\O |f(x)|^2 dx - \int_\O |\n f(x)|^2 dx = (B^{(N)} \vec c, \vec c),
\end{equation}
with $B^{(N)} = B_0^{(N)} + B_1^{(N)}$, 
$B_0^{(N)}$ satisfies \eqref{38},
$B_1^{(N)}$ satisfies \eqref{317}.
\end{itemize}
We apply Lemma \ref{l319} with
$$
a = \frac{\mes_d \O}{2 w(\O)^2}, \quad
b = \frac{32 (d-4)^{\frac{d-4}2} \oo_{d-2} (2\pi)^{d-2} \la^{\frac{d}2}}
{(d-1) d^{\frac{d}2} \oo_{d-1}^2 (\la-\de^2)^{d-3}}  
\left(\oo_{d-1} w(\O)^{d-2} \left(\mes_d\O\right)^d\right)^{\frac1{d-1}} .
$$
It means that we can find a number $N$, the set 
$H^{(N)} \in \left(S_{\sqrt{\la-\de^2}}^{d-2}\right)^N$,
and a subspace $F_* \subset F(H^{(N)})$ such that
the quadratic form \eqref{319} is positive on this subspace $F_*$, and
\begin{eqnarray*}
\dim F_* > \frac{a^2}{4b} 
= \frac{(\mes_d \O)^2}{16 w(\O)^4} \cdot
\frac{(d-1) d^{\frac{d}2} \oo_{d-1}^2 (\la-\de^2)^{d-3}}
{32 (d-4)^{\frac{d-4}2} \oo_{d-2} (2\pi)^{d-2} \la^{\frac{d}2}
\left(\oo_{d-1} w(\O)^{d-2} \left(\mes_d\O\right)^d\right)^{\frac1{d-1}}} \\
\ge \frac{d^3 (\la-\de^2)^{d-3} \oo_{d-1}^{2 - \frac1{d-1}} (\mes_d \O)^{\frac{d-2}{d-1}}}
{512 \la^{\frac{d}2} (2\pi)^{d-2} \oo_{d-2} w(\O)^{5 - \frac1{d-1}}}
= \frac{d^3 (\la-\de^2)^{d-3} \G\left(\frac{d}2\right) (\mes_d \O)^{\frac{d-2}{d-1}}}
{2^{d+7} \pi^{\frac{d-3}2} \la^{\frac{d}2} 
\G \left(\frac{d+1}2\right)^{2 - \frac1{d-1}} w(\O)^{5 - \frac1{d-1}}}.
\end{eqnarray*}
This estimate together with Lemma \ref{l31} gives

\begin{lemma}
\label{l320}
Let $d\ge 5$. 
Assume that a bounded domain $\O$ is contained in a layer of width $w(\O)$.
Assume that $\la \ge \de^2$ where $\de = w(\O)^{-1}$.
Then
\begin{equation}
\label{320}
N_{\cal N} (\O, \la) - N_{\cal D} (\O, \la) \ge
\frac{d^3 (\la-\de^2)^{d-3} \G\left(\frac{d}2\right) (\mes_d \O)^{\frac{d-2}{d-1}}}
{2^{d+7} \pi^{\frac{d-3}2} \la^{\frac{d}2} 
\G \left(\frac{d+1}2\right)^{2 - \frac1{d-1}} w(\O)^{5 - \frac1{d-1}}}.
\end{equation}
\end{lemma}
 
{\it Proof of Theorems \ref{t335} and \ref{t14}.}
By virtue of \eqref{15} it is sufficient to estimate the right hand side of \eqref{320}
from below for $\la \ge \la_1 (\O)$.
By the Faber-Krahn inequality
$$
\la_1 (\O) \ge \frac{j_{\frac{d}2-1}^2 \oo_d^{2/d}}{(\mes_d \O)^{2/d}},
$$
thus, in the rest of the proof we assume
\begin{equation}
\label{321}
\la \ge \frac{j_{\frac{d}2-1}^2 \oo_d^{2/d}}{(\mes_d \O)^{2/d}}.
\end{equation}

Let us find the asymptotics of the right hand side of \eqref{320}
if $d \to \infty$ up to terms of order $O (\log d)$.
Recall that 
$$
\de = \frac1{w(\O)} = \frac1{\left(d! \mes_d \O\right)^{1/d}},
$$
so,
$$
\frac{\de^2}\la \le \frac{(\mes_d \O)^{2/d}}
{\left(d! \mes_d \O\right)^{2/d} j_{\frac{d}2-1}^2 \oo_d^{2/d}}
= \frac{\G\left(\frac{d}2 + 1\right)^{2/d}}{\pi (d!)^{\frac2d} \,j_{\frac{d}2-1}^2}
= O \left(\frac1{d^3}\right), \quad d\to \infty.
$$
Therefore,
$$
\log \left(\frac{\la-\de^2}\la\right) = O \left(\frac1{d^3}\right),
$$
and
\begin{eqnarray}
\nonumber
\log \left(\frac{(\la-\de^2)^{d-3}}{\la^{\frac{d}2}}\right) 
= \left(\frac{d}2 - 3\right) \log \la + (d-3) \log \left(\frac{\la-\de^2}\la\right)\\
\label{322}
= \left(\frac{d}2 - 3\right) \log \la + O \left(\frac1{d^2}\right) \\
\ge \left(\frac{d}2 - 3\right)
\left(2 \log j_{\frac{d}2-1} + \frac2d \log \oo_d - \frac2d \log \mes_d \O\right)
+ O \left(\frac1{d^2}\right), \quad d\to \infty,
\nonumber
\end{eqnarray}
where we used \eqref{321}.

We have also
$$
\log \oo_d = \frac{d}2 \log \pi - \log \G \left(\frac{d}2 + 1\right)
= \frac{d}2 \log \pi - \frac{d}2 \left(\log d - \log 2 - 1\right) + O (\log d), \quad d\to \infty.
$$
Substituting this and \eqref{24} in \eqref{322} we obtain
\begin{eqnarray}
\nonumber
\log \left(\frac{(\la-\de^2)^{d-3}}{\la^{\frac{d}2}}\right) 
\ge (d-6) \left(\log j_{\frac{d}2-1} + \frac1d \log \oo_d - \frac1d \log \mes_d \O\right)
+ O \left(\frac1{d^2}\right) \\
\label{323}
= \frac{d \log d}2 + \frac{d}2 \left(\log \pi - \log 2 +1\right)
+ |a_1| 2^{1/3} d^{1/3} - \frac{d-6}d \, \log \mes_d \O + O (\log d), \\
d\to \infty.
\nonumber
\end{eqnarray}
Next,
\begin{equation}
\label{324}
\log \G \left(\frac{d}2\right) 
= \frac{d \log d}2 - \frac{d}2 \left(\log 2 +1\right) + O (\log d), 
\end{equation}
\begin{eqnarray}
\nonumber
\log\left(\G \left(\frac{d+1}2\right)^{2 - \frac1{d-1}}\right)
= \left(2 - \frac1{d-1}\right) \log \G \left(\frac{d+1}2\right) \\
\label{325}
= d \log d - d \left(\log 2 + 1\right) + O (\log d), \quad d \to \infty.
\end{eqnarray}
Finally, 
\begin{eqnarray}
\nonumber
\log \left(\frac{(\mes_d \O)^{\frac{d-2}{d-1}}}{w(\O)^{5 - \frac1{d-1}}}\right)
= \frac{d-2}{d-1} \log \mes_d \O 
- \frac{5d-6}{d(d-1)} \log \left(d! \mes_d \O\right) \\
= \frac{d-6}d \, \log \mes_d \O + O (\log d), \quad d \to \infty.
\label{326}
\end{eqnarray}
Substituting \eqref{323}, \eqref{324}, \eqref{325} and \eqref{326} into \eqref{320} we get
\begin{eqnarray*}
\log \Psi (d, k, \O) 
\ge  \frac{d \log d}2 + \frac{d}2 \left(\log \pi - \log 2 +1\right)
+ |a_1| 2^{1/3} d^{1/3} - \frac{d-6}d \, \log \mes_d \O \\
+ \frac{d \log d}2 - \frac{d}2 \left(\log 2 +1\right) 
- d \log 2 - \frac{d}2 \log \pi \\
- d \log d + d \left(\log 2 + 1\right) 
+ \frac{d-6}d \, \log \mes_d \O + O (\log d) \\
= \left(1 - \log 2\right) d + |a_1| 2^{1/3} d^{1/3} + O (\log d), \quad d \to \infty.
\end{eqnarray*}
Therefore,
$$
\log \Psi (d, k, \O) \ge \left(1 - \log 2\right) d
$$
for sufficiently large $d$.
Taking into account \eqref{*} we see that there is a constant $C_{conv} > 0$ such that
$$
\Psi (d,k,\O) \ge C_{conv} \left(\frac{e}2\right)^d \quad \text{for all}\ k\ \text{and}\ \O.
\qquad \qed
$$

%%%%%%%%%%%%%%%%%%%%%%%%%%%%%%%%%%%%%%%%%%
\section{Ball}
In this section we prove Theorem \ref{t16}.
We use some ideas from \cite{CMS}.

Let $\O = B_1 (0) \subset \R^d$ be a unit ball centered at the origine.
The first Dirichlet eigenvalue is 
$\la_1 = j_{\frac{d}2-1}^2$, see \eqref{21}.
The first eigenfunction of the Neumann problem is $\psi_1 (x) \equiv 1$
with the eigenvalue $\mu_1 = 0$.
Other eigenfunctions are
$$
\psi_{m,k} (x) = r^{1-\frac{d}2} J_{\frac{d}2 + m -1} \left(p'_{d,m,k} r\right) Y_m (\oo),
\quad m \in \N_0, \ k \in \N.
$$
Here $x = (r; \oo)$ are spherical coordinates in $\R^d$,
$Y_m$ are sphercial harmonics,
and $p'_{d,m,k}$ is the $k$-th positive root of the function
\begin{equation}
\label{41}
\left(r^{1-\frac{d}2} J_{\frac{d}2 + m -1} (r)\right)'
= r^{-\frac{d}2} \left(r J_{\frac{d}2 + m -1}' (r) 
+ \left(1-\frac{d}2\right) J_{\frac{d}2 + m -1} (r)\right).
\end{equation}
The corresponding eigenvalues are 
$$
\mu_{m,k} = \left(p'_{d,m,k}\right)^2
\quad \text{with multiplicity} \quad 
\ka_m := \left( \begin{array}{cc} m + d - 1 \\ d-1 \end{array} \right) 
- \left( \begin{array}{cc} m + d - 3 \\ d-1 \end{array} \right) ,
$$
where we assume $\left( \begin{array}{cc} p \\ q \end{array} \right) = 0$
if $p < q$.

By virtue of Dixon's Theorem \cite[\S 15.23]{Watson} the roots of the functions
$$
r J_{\frac{d}2 + m -1}' (r) + \left(1-\frac{d}2\right) J_{\frac{d}2 + m -1} (r)
\quad \text{and} \quad J_{\frac{d}2 + m -1} (r)
$$
interlace.
Thus, for all natural $m$ we have $p'_{d,m,2} > j_{\frac{d}2+m-1}$
due to \eqref{41}.
Next, $j_{\nu+1} > j_\nu$, see for example \cite[\S 15.22]{Watson},
and therefore, $p'_{d,m,2} > j_{\frac{d}2-1}$.
So, the set of eigenvalues $\mu_{m,k}$ that lesser than $\la_1$
consists of the eigenvalues of type $\mu_{m,1} = (p'_{d,m,1})^2$ only.
Note also that for $m=0$
$$
\left(r^{1-d/2} J_{\frac{d}2-1} (r)\right)' = - r^{-d/2} J_{\frac{d}2} (r),
$$
so, $p'_{d,0,1} = j_{\frac{d}2} > j_{\frac{d}2-1}$.
The numbers $p'_{d,m,1}$ increase in $m$ if $m \ge 1$.
Denote 
\begin{equation}
\label{42}
M = M(d) = \max \left\{m \in \N : p'_{d,m,1} \le j_{\frac{d}2-1}\right\}.
\end{equation}
Taking into account the first Neumann eigenvalue $\mu_1=0$ we get
$$
\Phi (d,1,B) = 1 + \sum_{m=1}^M \ka_m = \sum_{m=0}^M \ka_m
= \left( \begin{array}{cc} M + d - 1 \\ d-1 \end{array} \right)  
+ \left( \begin{array}{cc} M + d - 2 \\ d-1 \end{array} \right) .
$$
This equality implies
$$
\left( \begin{array}{cc} M + d - 1 \\ d-1 \end{array} \right) 
< \Phi (d,1,B) < 2 \left( \begin{array}{cc} M + d - 1 \\ d-1 \end{array} \right) ,
$$
and
\begin{equation}
\label{43}
\log \Phi (d,1,B) = \log \left( \begin{array}{cc} M + d - 1 \\ d-1 \end{array} \right) 
+ O (1), \quad d \to \infty.
\end{equation}

In order to find the asymptotics of $M(d)$ for large $d$ we need the following

\begin{lemma}[Lemma 2.4, \cite{mu3}]
If $m\ge 1$ and $d\ge 3$ then
\begin{equation}
\label{45}
m(m+d-2) \le \left(p'_{d,m,1}\right)^2 \le \frac{dm(m+d-2)}{d-2}.
\end{equation}
\end{lemma}

The asymptotics \eqref{23} yields
$$
\la_1 = j_{\frac{d}2-1}^2 = \frac{d^2}4 + |a_1| 2^{-2/3} d^{4/3} + O (d),
\quad d \to \infty.
$$
The equations
$$
m (m+d-2) = \la_1 \quad \text{and} \quad
m (m+d-2) = \frac{(d-2)\la_1}{d}
$$
have positive solutions
$$
m^* = \frac12 \left(\sqrt{4\la_1 + (d-2)^2} - d + 2\right)
\quad \text{and} \quad
m_* = \frac12 \left(\sqrt{\frac{4\la_1(d-2)}{d} + (d-2)^2} - d + 2\right)
$$
respectively.
Therefore, the number $M(d)$ defined by \eqref{42} satisfies the bounds
$$
m_*-1 < M (d) \le m^*.
$$
Furthermore,
$$
m^* = \frac12 \left(\sqrt{2 d^2 + |a_1| 2^{4/3} d^{4/3} + O (d)} - d + 2\right)
= \frac{\sqrt 2 - 1}2 \, d + |a_1| 2^{-7/6} d^{1/3} + O (1), \quad d \to \infty.
$$
In the same way,
$$
m_* = \frac{\sqrt 2 - 1}2 \, d + |a_1| 2^{-7/6} d^{1/3} + O (1), \quad d \to \infty,
$$
and thus, the same is true for $M$,
$$
M (d) = \frac{\sqrt 2 - 1}2 \, d + |a_1| 2^{-7/6} d^{1/3} + O (1), \quad d \to \infty.
$$
Clearly,
\begin{equation}
\label{46}
\log \left( \begin{array}{cc} M + d - 1 \\ d-1 \end{array} \right) 
= \log \left((M+d-1)!\right) - \log \left(M!\right) - \log\left((d-1)!\right).
\end{equation}

The Stirling formula implies the following

\begin{lemma}
\label{l42}
If 
$$
K(d) = \ga_0 d + \ga_1 d^{1/3} + O(1), \quad d \to \infty, \quad \ga_0 > 0,
$$
then
$$
\log \left(K(d)!\right) = \ga_0 d \log d + \ga_0 (\log \ga_0 -1) d 
+ \ga_1 d^{1/3} \log d + \ga_1 \log \ga_0 d^{1/3} + O (\log d), \quad d \to \infty.
$$
\end{lemma}

Therefore,
\begin{eqnarray}
\nonumber
\log \left((M+d-1)!\right) 
= \frac{\sqrt 2 + 1}2\, d \log d + \frac{\sqrt 2 + 1}2 \left(\log \frac{\sqrt 2 + 1}2 - 1\right) d \\
+ |a_1| 2^{-7/6} d^{1/3} \log d + |a_1| 2^{-7/6} \log \frac{\sqrt 2 + 1}2 d^{1/3} + O(\log d),
\label{47}
\end{eqnarray}
\begin{eqnarray}
\nonumber
\log \left(M!\right) 
= \frac{\sqrt 2 - 1}2\, d \log d + \frac{\sqrt 2 - 1}2 \left(\log \frac{\sqrt 2 - 1}2 - 1\right) d \\
+ |a_1| 2^{-7/6} d^{1/3} \log d + |a_1| 2^{-7/6} \log \frac{\sqrt 2 - 1}2 d^{1/3} + O(\log d),
\label{48}
\end{eqnarray}
and
\begin{equation}
\label{49}
\log \left((d-1)!\right)
= d \log d - d + O (\log d), \quad d \to \infty.
\end{equation}
Substituting \eqref{46}, \eqref{47}, \eqref{48}, \eqref{49} into \eqref{43} we get
\begin{equation}
\label{410}
\log \Phi (d,1,B) 
= \left(\sqrt 2 \log (\sqrt 2 + 1) - \log 2\right) d
+ |a_1| 2^{-1/6} \log (\sqrt 2 + 1) d^{1/3} + O (\log d), \quad d \to \infty,
\end{equation}
where $a_1$ is the first zero of the Airy function.
Theorem \ref{t16} is proved.
\qed

\section{Cube}
In this section we prove Theorem \ref{t17}.

\subsection{Proof of Theorem \ref{t17} a)}
Let $\O = Q = (0,\pi)^d$ be a $d$-dimensional cube.
The eigenfunctions and eigenvalues of the Dirichlet problem in $Q$ are
$$
\ph_{\vec m} (x) = \sin m_1 x_1 \sin m_2 x_2 \dots \sin m_d x_d,
\quad \la_{\vec m} = m_1^2 + \dots + m_d^2, \quad \vec m \in \N^d.
$$
Thus, the first Dirichlet eigenvalue is $\la_1 (Q) = d$.

The eigenfunctions and eigenvalues of the Neumann problem in $Q$ are
$$
\psi_{\vec m} (x) = \cos m_1 x_1 \cos m_2 x_2 \dots \cos m_d x_d,
\quad \mu_{\vec m} = m_1^2 + \dots + m_d^2, \quad \vec m \in \N_0^d.
$$
Thus, 
$$
\Phi (d,1,Q) = \# \left\{\vec m \in \N_0^d : \vec m^2 \le d\right\}, \quad
\Psi (d,1,Q) = \# \left\{\vec m \in \N_0^d : \vec m^2 \le d\right\} - 1.
$$
If for some natural $k$
$$
\la_k = \la_{\vec n} = \vec n^2, \qquad \vec n \in \N^d,
$$
then
$$
k \le \#  \left\{\vec m \in \N^d : \vec m^2 \le \vec n^2\right\},
$$
and
$$
\Psi (d,k,Q) = \Phi (d,k,Q) - k
\ge \#  \left\{\vec m \in \N_0^d : \vec m^2 \le \vec n^2\right\}
- \#  \left\{\vec m \in \N^d : \vec m^2 \le \vec n^2\right\}.
$$
The difference in the right hand side coincides with the number
of integer vectors $\vec m \in \N_0^d$ such that
$\vec m^2 \le \vec n^2$, and at least one coordinate vanishes, $m_j=0$ for some $j$.
This set contains the set
$$
\left\{\vec m \in \N_0^d : \vec m^2 \le d\right\} \setminus \{(1,1,\dots,1)\}.
$$
So,
$$
\Psi (d,k,Q) \ge 
\# \left(\left\{\vec m \in \N_0^d : \vec m^2 \le d\right\} \setminus \{(1,1,\dots,1)\}\right)
= \Psi (d,1,Q).
$$
The claim a) of Theorem \ref{t17} is proved.

\subsection{Proof of Theorem \ref{t17} b)}
Now, we need to describe the behaviour of $\Phi (d,1,Q)$,
i.e. the number of solutions of the inequality
\begin{equation}
\label{51}
m_1^2 + \dots +m_d^2 \le d,
\qquad m_j \in \N_0,
\end{equation}
in non-negative integeres.
A large amount of number-theoretical papers is devoted 
to the study of the number of integer points in a large ball
or in another body.
Usually, if the parameter of dilation tends to infinity,
the number of integer points in such body grows like its volume asymptotically.
The peculiarity of our case is that the square of the radius of the ball 
and the dimension of space tend to infinity simultaneously.
The general principle does not work, and the principal term of asymptotics
{\it does not coincide} with the volume of dilated set.
This phenomenon was discovered in \cite{MO}.
In particular, in this paper the asymptotics of number of solutions to \eqref{51}
in integers (i.e. $m_j \in \Z$, not $m_j \in \N_0$) is found,
see formula (3.4) in \cite{MO}.
We use the same approach.

Introduce the generating function
$$
\te (z) := 1 + z + z^4 + z^9 + \dots = \sum_{j=0}^\infty z^{j^2}.
$$
It is well defined and analytic in the circle $|z|<1$.
If we raise this series to power $d$,
\begin{equation}
\label{52}
\te (z)^d = \sum_{k=0}^\infty r(k) z^k,
\end{equation}
then the coefficient $r(k)$ in front of $z^k$ coincides with the number of solutions
to the equation
$$
m_1^2 + \dots +m_d^2 = k,
\qquad m_j \in \N_0.
$$
Equality \eqref{52} implies
$$
r(k) = \frac1{2\pi i} \oint_{|z|=\tau} \frac{\te(z)^d dz}{z^{k+1}},
$$
where $\tau$ is a real number from $(0,1)$.
Therefore,
\begin{eqnarray}
\label{53}
\Phi (d,1,Q) = \# \left\{\vec m \in \N_0^d : \vec m^2 \le d\right\} 
= \sum_{k=0}^d r(k) \\
= \frac1{2\pi i} \oint_{|z|=\tau} \te(z)^d \left(1 + z + z^2 + \dots + z^d\right) \frac{dz}{z^{d+1}}
= \frac1{2\pi i} \oint_{|z|=\tau} \frac{\te(z)^d dz}{z^{d+1} (1-z)},
\nonumber
\end{eqnarray}
as $\oint_{|z|=\tau} \te(z)^d z^n dz = 0$ if $n \ge 0$.
It remains to find the asymptotics of the integral
\begin{equation}
\label{54}
\frac1{2\pi i} \oint_{|z|=\tau} \left(\frac{\te(z)}{z}\right)^d \frac{dz}{z(1-z)}
\end{equation}
for large $d$ using the saddle-point method.

Let us consider the function
\begin{equation}
\label{55}
\psi (z) : = \log \te (z) - \log z
\end{equation}
for real $z \in (0,1)$.
Clearly,
$$
\psi' (z) = \frac{\te'(z)}{\te(z)} - \frac1z.
$$
The equation $\te' (z) z = \te (z)$ is equivalent to the equation
\begin{equation}
\label{56}
\sum_{j=2}^\infty (j^2-1) z^{j^2} = 1.
\end{equation}
The left hand side is strictly increasing on $[0,1]$,
so this equation has a unique solution 
$$
x_0 \approx 0,689.
$$
Thus, the function $\te (z) z^{-1}$ has a unique point of minimum on $(0,1)$.
We take this value $x_0$ as a radius in the integral \eqref{54},
$$
\Phi (d,1,Q) = \frac1{2\pi i} \oint_{|z|=x_0} \left(\frac{\te(z)}{z}\right)^d \frac{dz}{z(1-z)}.
$$
Represent the points of a circle of integration as $z = x_0 e^{i\ph}$, $|\ph| \le \pi$.
If $|\ph| \ge \de > 0$, then
$$
\sqrt{1+2x_0\cos\ph+x_0^2} \le \sqrt{1+2x_0\cos\de+x_0^2} \le 1 + x_0 - \er,
$$
with some $\er = \er(\de) > 0$.
Therefore,
$$
|\te(z)| \le \left|1 + x_0 e^{i\ph}\right| + \sum_{j=2}^\infty x_0^{j^2}
\le \te(x_0) - \er
$$
for $z = x_0 e^{i\ph}$, $|\ph| \ge \de$, and
$$
\left|\frac{\te(z)}z\right|^d 
\le \left(\frac{\te(x_0)}{x_0}\right)^d \left(1 - \frac\er{\te(x_0)}\right)^d.
$$
Thus, the contribution to the integral from the entire circle, 
excluding a small neighborhood of the point $z=x_0$, 
is exponentially small compared to the contribution of that neighborhood.

The function \eqref{55} is well defined in a complex neighborhood of the point $z=x_0$
(we choose the branch of the logarithm in such a way that its value is real at $x_0$).
We have
$$
\psi (z) = \psi (x_0) + \frac{\psi''(x_0)}2 \,(z-x_0)^2 
+ \frac{\psi'''(x_0)}6 \,(z-x_0)^3 + \dots,
$$
$\psi'' (x_0) \approx 6,38$.
Put $z = x_0 + \xi + i\eta$, $\xi, \eta \in \R$.
Then the set of such $z$ in a neighborhood of $x_0$
where $\im \psi (z) = 0$ consists of the direct line $\eta = 0$
and of a curve $\xi = \xi_0 (\eta)$.
We have
$$
\im \psi (z) = \psi'' (x_0) \xi \eta + \frac{\psi'''(x_0)}6 \,(3 \xi^2 \eta - \eta^3) + \dots,
$$
therefore,
\begin{equation}
\label{555}
\xi_0 (\eta) = O (\eta^2), \quad \eta \to 0.
\end{equation}
Note that on the initial circle $|x_0 + \xi + i \eta| = x_0$ 
we also have $\xi = O (\eta^2)$, $\eta \to 0$.
Now, we choose a contour $\G$ of integration to consist of four parts:
$$
\G = \G_0 \cup \G_1 \cup \G_2 \cup \G_3.
$$
Here
\begin{equation*}
%\label{57}
\G_0 = \left\{ z = x_0 + \xi_0 (\eta) + i \eta \right\}_{|\eta| < \de_1}
\end{equation*}
is a part of a curve $\xi = \xi_0(\eta)$ in a neighborhood of $x_0$,
a sufficiently small parameter $\de_1$ will be chosen below;
$\G_1$ and $\G_2$ are symmetric radial segments connecting this curve
with a circle $|z| = x_0$; and $\G_3$ is the remaining part of the circle,
$$
\G_3 = \left\{ z \in \C : |z| = x_0, |\arg z| \ge \arctan \frac{\de_1}{x_0 + \xi_0 (\de_1)}\right\}.
$$
We have $\xi = O (\eta^2)$, $\eta \to 0$, on the radial segments $\G_1$ and $\G_2$.
Therefore,
\begin{eqnarray*}
\re \psi(z) =  \psi (x_0) + \frac{\psi''(x_0)}2 \,(\xi^2-\eta^2) 
+ \frac{\psi'''(x_0)}6 \,(\xi^3 - 3 \xi \eta^2) + \dots \\
= \psi (x_0) - \frac{\psi''(x_0)}2\, \eta^2 + O (\eta^4), \qquad \eta \to 0,
\end{eqnarray*}
there.
Choose the parameter $\de_1$ so small that
$$
\re \psi(z) \le \psi (x_0) - \frac{\psi''(x_0)}4\, \de_1^2 
\quad \text{on} \quad \G_1 \cup \G_2.
$$
Then the contribution to the integral from these segments  
is also exponentially small compared to the contribution of the curve $\G_0$.

Finally, we have
$$
\frac1{2\pi i} \int_{\G_0} \left(\frac{\te(z)}{z}\right)^d \frac{dz}{z(1-z)}
= \frac1{2\pi i} \int_{-\de_1}^{\de_1} 
\frac{e^{\tilde \psi (\eta) d}\left(\xi_0' (\eta) + i\right) d \eta}
{\left(x_0 + \xi_0 (\eta) + i \eta\right) \left(1-x_0-\xi_0 (\eta)-i\eta\right)},
$$
where
$$
\tilde \psi(\eta) := \psi (x_0 + \xi_0 (\eta) + i \eta)
= \psi (x_0) + \frac{\psi''(x_0)}2 \,(\xi(\eta)^2-\eta^2) + \dots
$$
Note that
\begin{equation}
\label{58}
\tilde \psi (0) = \psi (x_0), \quad \tilde \psi' (0) = 0,
\quad \tilde \psi'' (0) = - \psi'' (x_0),
\end{equation}
and $\eta = 0$ is the unique point on $[-\de_1, \de_1]$ where
the first derivative $\tilde \psi'$ vanishes.
Therefore,
\begin{eqnarray*}
\frac1{2\pi i} \int_{-\de_1}^{\de_1} 
\frac{e^{\tilde \psi (\eta) d}\left(\xi_0' (\eta) + i\right) d \eta}
{\left(x_0 + \xi_0 (\eta) + i \eta\right) \left(1-x_0-\xi_0 (\eta)-i\eta\right)} \\
= \frac1{2\pi i} \frac{\sqrt{2\pi} e^{\tilde \psi (0) d}}{\sqrt{-\tilde \psi'' (0) d}}
\left.\frac{\xi_0' (\eta) + i}
{\left(x_0 + \xi_0 (\eta) + i \eta\right) \left(1-x_0-\xi_0 (\eta)-i\eta\right)}\right|_{\eta=0} 
\left(1 + O\left(\frac1d\right)\right) \\
= \frac{e^{\psi (x_0) d}}{\sqrt{2\pi \psi''(x_0) d} \ x_0 (1-x_0)} 
\left(1 + O\left(\frac1d\right)\right), \quad d \to \infty,
\end{eqnarray*}
where we used the Laplace method and the relations \eqref{555} and \eqref{58}.
The claim b) of Theorem \ref{t17} follows.
The constants are 
\begin{equation}
\label{consts}
\al_{cube} = \psi (x_0) \approx 1,04, \qquad
\be =  \frac1{\sqrt{2\pi \psi''(x_0)} \ x_0 (1-x_0)}
\approx 0,74.
\end{equation}

%%%%%%%%%%%%%%%%%%%%%%%%%%%%%%%%%%%%%%%%%%

\end{document}